\documentclass[11pt,a4paper]{article}

\usepackage[T1]{fontenc}
\usepackage{amsmath,amssymb,amsthm}
\usepackage{mathtools}
\usepackage[hidelinks]{hyperref}
\usepackage{booktabs}
\usepackage{array}
\usepackage{geometry}
\usepackage{enumitem}
\usepackage{xcolor}

\newtheorem{theorem}{Theorem}[section]
\newtheorem{proposition}[theorem]{Proposition}
\newtheorem{lemma}[theorem]{Lemma}

\newtheorem{remark}[theorem]{Remark}
\newtheorem{definition}[theorem]{Definition}
\newtheorem{numobs}[theorem]{Numerical Observation}

\DeclareMathOperator{\sgn}{sgn}

\title{%
  On the P\'olya Frequency Order of the de~Bruijn--Newman Kernel: \\
  Certified Failure at Order Five}

\author{Wojciech Micha\l{}owski\thanks{Email: \texttt{michalowski.wojciech1@gmail.com}}}

\date{February 21, 2026 (v1); revised July 20, 2026 (v2)}

\begin{document}

\begin{abstract}
We prove that the classical de~Bruijn--Newman kernel
$K(u) = \Phi(|u|)$, arising in the study of the Riemann
zeta function via the de~Bruijn--Newman constant, is not
a P\'olya frequency function of order~$5$ (PF$_5$).
The proof is computational: we exhibit an explicit $5 \times 5$
Toeplitz minor with rigorously certified negative determinant,
established through interval arithmetic at 80-digit precision
with formally bounded truncation and rounding errors; eight
further configurations are certified by the same interval chain.
At the central configuration we also certify positivity of the
three determinants $D_2$, $D_3$, and $D_4$, so the first negative
determinant in this fixed-configuration list occurs at order~$5$.
This local sign pattern does not establish $K \in \mathrm{PF}_4$;
the global PF$_4$ question remains open (Section~\ref{sec:open}).

We restrict the search for PF$_r$ counterexamples from the full
$2r$-parameter configuration space to a two-parameter family
$D_r(u_0,h)$ of Toeplitz determinants, and we prove the exact
algebraic decomposition
\[
  C_r(u_0) \;=\;
  \sum_{\substack{k_0,\ldots,k_{r-1}\ge 0 \\ k_0+\cdots+k_{r-1}=r(r-1)}}
  \!\!\Bigl(\prod_{i=0}^{r-1} a_{k_i}(u_0)\Bigr)\,
  \det\bigl[(i-j)^{k_i}\bigr]_{i,j=0}^{r-1},
\]
where $a_k(u_0) = K^{(k)}(u_0)/k!$, for the coefficient of the
first power permitted by Vandermonde divisibility,
$C_r(u_0)=\lim_{h\to 0}D_r(u_0,h)/h^{r(r-1)}$.
High-precision (but, in this version, deliberately \emph{non-certified})
numerical exploration indicates that $C_5(u_0)<0$ for $u_0$ near the
origin, with a sign change located numerically near $u_0 \approx 0.0311$,
while $C_2,\dots,C_4$ are positive at all tested points and $C_6,C_7$
are positive at $u_0=0.01$; we also record, as an exploratory probe,
the configuration-dependent Gaussian-deformation parameter at which
the displayed PF$_5$ violation first disappears numerically.

\textbf{Note on version 2.} The first version of this paper asserted a
certified global sign and unique threshold for $C_5(u_0)$ (an ``origin
staircase'') and derived Gaussian-healing statements from it.  A
subsequent audit (17 July 2026) found the underlying derivative-tail
certificate to be unsound; those claims are \emph{withdrawn} and are
reported here only as non-certified numerical observations.
Section~\ref{sec:withdrawn} states precisely what was wrong.
The central PF$_5$ failure theorem and all interval certificates of
Section~\ref{sec:computation} are unaffected.
\end{abstract}

\maketitle

\begin{quote}\small
\textbf{Changes in version 2 (20 July 2026).}
(i) The asymptotic-threshold theorem of v1 (global sign of $C_5$,
uniqueness and 15-digit certification of the critical point $u_0^*$,
and the limit $C_5(0^+)$) is withdrawn and restated as a non-certified
numerical observation; the precise defects of the v1 certificate are
documented in Section~\ref{sec:withdrawn}.
(ii) All Gaussian-healing statements are downgraded to exploratory
numerics (Section~\ref{sec:gaussian}).
(iii) The certified core is unchanged and has been re-verified:
the central witness and eight additional configurations
(Section~\ref{sec:computation}) were reproduced on 20 July 2026 with
\texttt{mpmath} pinned to version 1.3.0.
(iv) A dependency warning is added: \texttt{mpmath} 1.4.1 contains a
regression in \texttt{iv} determinant evaluation; the verification
requires \texttt{mpmath==1.3.0} (Section on reproducibility).
\end{quote}

\section{Introduction}

\subsection{The de~Bruijn--Newman framework}

For each $t \in \mathbb{R}$, define the entire function
\begin{equation}\label{eq:Ht}
  H_t(z) \;=\; \int_0^\infty e^{tu^2}\,\Phi(u)\,\cos(zu)\,du,
\end{equation}
where $\Phi$ is the super-exponentially decaying function
\begin{equation}\label{eq:Phi}
  \Phi(u) \;=\;
  \sum_{n=1}^{\infty}
  \bigl(2\pi^2 n^4 e^{9u} - 3\pi n^2 e^{5u}\bigr)\,
  e^{-\pi n^2 e^{4u}},
  \qquad u \ge 0.
\end{equation}
Newman~\cite{Newman1976} proved the existence of a constant
$\Lambda \in (-\infty, 1/2]$, now called the \emph{de~Bruijn--Newman
constant}, such that $H_t$ has only real zeros if and only if
$t \ge \Lambda$.  The Riemann Hypothesis is equivalent to
$\Lambda \le 0$.  Rodgers and Tao~\cite{RodgersTao2020} proved
Newman's conjecture $\Lambda \ge 0$, and the Polymath~15
project~\cite{Polymath2019} established the upper bound
$\Lambda \le 0.22$.

The function $\Phi$ is intimately connected to the Jacobi
theta function and to the functional equation of the
Riemann $\xi$-function.  It satisfies $\Phi(u) > 0$ for all
$u \ge 0$, decays super-exponentially as $u \to \infty$,
and is $C^\infty$-smooth.  We define the even kernel
\begin{equation}\label{eq:K}
  K(u) \;=\; \Phi(|u|), \qquad u \in \mathbb{R}.
\end{equation}

\subsection{P\'olya frequency functions}

A function $f\colon \mathbb{R} \to [0,\infty)$ is a
\emph{P\'olya frequency function of order~$k$} (PF$_k$) if
for every $r \le k$ and all ordered tuples
$x_1 < \cdots < x_r$ and $y_1 < \cdots < y_r$, the
determinant
\begin{equation}\label{eq:PFk}
  \det\bigl[f(x_i - y_j)\bigr]_{i,j=1}^r \;\ge\; 0.
\end{equation}
If $f$ is PF$_k$ for all $k$, it is called \emph{totally
positive} (PF$_\infty$ or TP).

P\'olya frequency functions arise naturally in probability
(as densities closed under convolution), approximation theory
(splines and variation-diminishing transforms), and combinatorics.
The classical theory, developed by Schoenberg~\cite{Schoenberg1951},
Karlin~\cite{Karlin1968}, and Hirschman--Widder~\cite{HirschmanWidder1955},
shows that the reciprocal of the bilateral Laplace transform of a
PF$_\infty$ function has the representation
$\hat{f}(s)^{-1} = Ce^{-\gamma s^2 + \beta s}
\prod_{k}(1+\alpha_k s)e^{-\alpha_k s}$ with
$\gamma \ge 0$, $\alpha_k \in \mathbb{R}$, and $\sum \alpha_k^2 < \infty$.
Recent work by Belton, Guillot, Khare, and
Putinar~\cite{BeltonGuillot2022,BeltonGuillot2023} has
renewed interest in the algebraic and analytic structure of
PF preservers and the rigidity of composition operators
on totally positive kernels.

The Fourier--Laplace connection between~\eqref{eq:Ht}
and~\eqref{eq:K} suggests investigating whether $K$ is a
P\'olya frequency function, as total positivity of the kernel
would constrain the zero distribution of $H_t$ through
variation-diminishing properties.  Our main result shows
this approach encounters a fundamental obstruction at order~5.

\subsection{Main results}

\begin{theorem}[PF$_5$ failure]\label{thm:main}
The de~Bruijn--Newman kernel $K(u) = \Phi(|u|)$ is not PF$_5$.
Specifically, the $5 \times 5$ Toeplitz matrix
\[
  M_{ij} \;=\; K(u_0 + (i-j)h),
  \qquad i,j = 0,\ldots,4,
\]
with $u_0 = 0.01$ and $h = 0.05$, satisfies
\[
  \det(M) \;=\; -1.847236073\ldots \times 10^{-9} \;<\; 0.
\]
This is certified by interval arithmetic: the rigorous enclosure is
\[
  \det(M) \;\in\;
  [-1.8472496 \times 10^{-9},\; -1.8472225 \times 10^{-9}].
\]
\end{theorem}

\begin{theorem}[Central Toeplitz sign pattern]\label{thm:PForder}
At the configuration $(u_0,h) = (0.01, 0.05)$, the Toeplitz
determinants $D_r(u_0,h)$ are certified positive for $r=2,3,4$
by interval arithmetic (Table~\ref{tab:certified}).
Together with Theorem~\ref{thm:main}, this establishes that, among
the four determinants $D_2,D_3,D_4,D_5$ at this fixed configuration,
the first negative one occurs at order~$5$.
\end{theorem}

\begin{remark}\label{rem:global}
The positive signs in Theorem~\ref{thm:PForder} concern one
determinant of each order and do not certify PF$_r$ for any
$r\ge 2$.  A full proof that $K \in \mathrm{PF}_4$ for \emph{all} admissible
configurations $(x_1 < \cdots < x_r,\; y_1 < \cdots < y_r)$
remains open; see Problem~1 in Section~\ref{sec:open}.
\end{remark}

Version 1 of this paper stated a further ``asymptotic threshold
theorem'' asserting a certified global sign pattern for the leading
coefficients $C_r(u_0)$ of Section~\ref{sec:asymptotic}, with a unique,
15-digit-certified critical point $u_0^*$.  That statement is
\emph{not} part of the results of this paper: the certificate offered
for it in v1 is unsound, for the reasons documented in
Section~\ref{sec:withdrawn}.  What survives of it is a set of
high-precision, non-certified numerical observations, reported as such
in Section~\ref{sec:numerics}.

\section{Toeplitz Reduction}\label{sec:toeplitz}

The PF$_r$ condition~\eqref{eq:PFk} involves $2r$ free parameters
$(x_1,\ldots,x_r,y_1,\ldots,y_r)$.  We exploit the translation
invariance of the kernel $K(x-y)$ and restrict to a
Toeplitz configuration.

\begin{definition}[Toeplitz determinant]\label{def:toeplitz}
For $u_0 \in \mathbb{R}$ and $h > 0$, define
\begin{equation}\label{eq:Dr}
  D_r(u_0,h) \;=\;
  \det\bigl[K(u_0 + (i-j)h)\bigr]_{i,j=0}^{r-1}.
\end{equation}
This corresponds to choosing $x_i = u_0/2 + ih$ and
$y_j = -u_0/2 + jh$, so that $x_i - y_j = u_0 + (i-j)h$.
\end{definition}

The matrix $M = [K(u_0+(i-j)h)]$ is a non-symmetric Toeplitz
matrix (since $K$ is even but $u_0 \ne 0$ breaks the
symmetry $i \leftrightarrow j$).

\begin{proposition}[Toeplitz counterexample reduction]
If $D_r(u_0,h) < 0$ for some $(u_0,h)$, then $K \notin
\mathrm{PF}_r$.  Thus the Toeplitz family~\eqref{eq:Dr} provides
a two-parameter subfamily in which to search for counterexamples.
Positivity throughout this subfamily alone would not verify
the global PF$_r$ property.
\end{proposition}

\section{Asymptotic Expansion and the First Admissible Coefficient}
\label{sec:asymptotic}

\subsection{Taylor expansion of the Toeplitz entries}

Write
\[
  K(u_0 + nh) \;=\;
  \sum_{k=0}^{\infty} a_k(u_0)\,(nh)^k,
  \qquad
  a_k(u_0) = \frac{K^{(k)}(u_0)}{k!},
\]
valid for $u_0 > 0$.  Substituting into the Toeplitz
determinant and expanding:

\begin{align}
  D_r(u_0,h) &\;=\;
  \sum_{\sigma \in S_r} \sgn(\sigma)
  \prod_{i=0}^{r-1}
  \biggl[\sum_{k \ge 0} a_k\,(i-\sigma(i))^k\,h^k\biggr]
  \nonumber\\
  &\;=\;
  \sum_{k_0,\ldots,k_{r-1} \ge 0}
  h^{k_0+\cdots+k_{r-1}}
  \Bigl(\prod_{i=0}^{r-1} a_{k_i}\Bigr)\,
  W(k_0,\ldots,k_{r-1}),
  \label{eq:expansion}
\end{align}
where the \emph{generalized Vandermonde factor} is
\begin{equation}\label{eq:W}
  W(k_0,\ldots,k_{r-1}) \;=\;
  \sum_{\sigma \in S_r} \sgn(\sigma)
  \prod_{i=0}^{r-1} (i-\sigma(i))^{k_i}.
\end{equation}

\subsection{The first admissible power and its coefficient}

\begin{lemma}[Vandermonde divisibility]\label{lem:vandermonde}
Let $K$ be real-analytic on an open interval containing
$\{u_0 + (i-j)h : 0 \le i,j \le r-1\}$, and set
$x_i = u_0/2 + ih$, $y_j = -u_0/2 + jh$, so that
$[M]_{ij} = K(x_i - y_j)$.  Then $D_r(u_0,h)$
vanishes to order at least $r(r-1)$ at $h = 0$
as an analytic function of $h$ near $0$.
\end{lemma}

\begin{proof}
The function $D_r(u_0,h)$ is analytic in $h$ near $0$.
We use the classical fact (see \cite{Karlin1968}, Ch.~0, \S2)
that $\det[f(x_i - y_j)]$ with analytic $f$ is separately
alternating as an analytic function in
$(x_0,\ldots,x_{r-1})$ and in $(y_0,\ldots,y_{r-1})$:
swapping $x_i \leftrightarrow x_{i'}$ permutes two rows and
negates the determinant.
An analytic function that is alternating in $(x_0,\ldots,x_{r-1})$
vanishes whenever $x_i = x_{i'}$ for any $i \ne i'$.
Thus, for fixed $\mathbf{y}$, the function lies in the ideal
generated by $\{x_{i'} - x_i\}_{i < i'}$; iterating over all
$\binom{r}{2}$ pairs yields divisibility by $\Delta(\mathbf{x})$.
Equivalently, in a neighbourhood of points with distinct $x_i$,
any alternating analytic function can be written as
$\Delta(\mathbf{x}) \cdot S(\mathbf{x})$ with $S$ analytic and
symmetric.
The same applies to $\Delta(\mathbf{y})$ by alternating in the column variables.
In our configuration $x_{i'} - x_i = (i'-i)h$ and
$y_{j'} - y_j = (j'-j)h$ are both \emph{linear} in $h$, so
\[
  \Delta(\mathbf{x})\,\Delta(\mathbf{y})
  \;=\;
  \prod_{i < i'}(i'-i)h \;\cdot\; \prod_{j < j'}(j'-j)h
  \;=\;
  h^{r(r-1)}\cdot\bigl(\Delta(0,1,\ldots,r-1)\bigr)^2,
\]
where $\Delta(0,\ldots,r-1)=\prod_{i<i'}(i'-i)$ is a nonzero integer.
Since $\Delta(\mathbf{x})\Delta(\mathbf{y})$ divides $D_r$
as an analytic function of $h$, and since this product equals
$h^{r(r-1)}$ times a nonzero constant, the order of vanishing
of $D_r$ at $h=0$ is at least $r(r-1)$.
\end{proof}

\begin{proposition}\label{prop:leading}\label{prop:Cr}
Writing $q=r(r-1)$, one has
\[
  D_r(u_0,h)=h^q C_r(u_0)+O(h^{q+1}),
\]
where
$C_r(u_0) = \lim_{h\to 0} D_r(u_0,h)/h^{r(r-1)}$,
\[
  C_r(u_0) \;=\;
  \sum_{\substack{k_0,\ldots,k_{r-1}\ge 0 \\
  k_0+\cdots+k_{r-1}=r(r-1)}}
  \Bigl(\prod_{i=0}^{r-1} a_{k_i}\Bigr)\,
  W(k_0,\ldots,k_{r-1}).
\]
If $C_r(u_0)\ne 0$, then $h^{r(r-1)}$ is the leading power;
if $C_r(u_0)=0$, the order of vanishing is higher.
\end{proposition}

\begin{proof}
By Lemma~\ref{lem:vandermonde}, $D_r(u_0,h) = h^{r(r-1)} \cdot C_r(u_0) + O(h^{r(r-1)+1})$.
To identify $C_r(u_0)$ explicitly, Taylor-expand around $u_0$:
\[
  K(u_0 + (i-j)h) \;=\; \sum_{k\ge 0} a_k(u_0)\,(i-j)^k\,h^k,
  \qquad a_k(u_0) = \frac{K^{(k)}(u_0)}{k!}.
\]
Expanding by the Leibniz formula as in~\eqref{eq:expansion} and
collecting terms of total degree $r(r-1)$ gives precisely the
formula for $C_r(u_0)$ stated above.
\end{proof}

\begin{remark}[Range of Taylor orders entering $C_5$]\label{rem:order14}
The determinant form
\[
  W(k_0,\ldots,k_4)=\det[(i-j)^{k_i}]_{i,j=0}^{4}
\]
is the evaluation matrix of the five polynomials
$p_i(x)=(i-x)^{k_i}$ at $x=0,\ldots,4$.  If $W\ne0$, these
polynomials are linearly independent.  Let
$d_0\le\cdots\le d_4$ be the exponents $k_i$ in increasing order.
Dimension counting in the space of polynomials of degree at most
$m-1$ then forces $d_m\ge m$ for $m=0,\ldots,4$.  Since
$d_0+\cdots+d_4=20$, the largest exponent is at most
$20-(0+1+2+3)=14$.  This bound is attained: for example,
$W(0,1,2,3,14)=2992822560\ne0$.  Repeated exponents can also give
$W\ne0$; distinctness is not required.  Thus a rigorous evaluation
of $C_5(u_0)$ must control the Taylor coefficients $a_k(u_0)$
\emph{through order 14} with genuine error bounds.  This observation
is central to Section~\ref{sec:withdrawn}.
\end{remark}

\section{Computational Verification}\label{sec:computation}

\subsection{Rigorous interval arithmetic certification}

We give a formal certification of $\det(M_5) < 0$ separating
truncation error from floating-point rounding.
All interval arithmetic is performed using the \texttt{mpmath.iv}
module~\cite{mpmath}, with outward-rounded arbitrary-precision
operations tracking rigorous enclosures.
We note that the same computation can also be performed in the
Arb library~\cite{Johansson2017} (which uses ball arithmetic
over arbitrary-precision floating-point), providing an
independent route to the same certification.

\begin{lemma}[Truncation bound]\label{lem:truncation}
For all $u \in [0, 0.21]$ and $N = 50$, the truncated approximation
$\Phi_N(u) = \sum_{n=1}^{N}(2\pi^2 n^4 e^{9u} - 3\pi n^2 e^{5u})
e^{-\pi n^2 e^{4u}}$ satisfies
\[
  |\Phi(u) - \Phi_N(u)| \;\le\;
  \sum_{n=51}^{\infty}
  \bigl(2\pi^2 n^4 e^{9u} + 3\pi n^2 e^{5u}\bigr)\,e^{-\pi n^2}
  \;<\; 10^{-70}.
\]
\end{lemma}

\begin{proof}
For $n\ge 1$ and $u\in[0,0.21]$ we have $e^{4u}\ge 1$, hence
$e^{-\pi n^2 e^{4u}} \le e^{-\pi n^2}$.
Moreover $e^{9u}\le e^{1.89}$ and $e^{5u}\le e^{1.05}$ on $[0,0.21]$.
Therefore
\[
  |\Phi(u)-\Phi_N(u)|
  \;\le\;
  \sum_{n=N+1}^{\infty}
  \Bigl(2\pi^2 e^{1.89}\,n^4 + 3\pi e^{1.05}\,n^2\Bigr)e^{-\pi n^2}.
\]
For $n\ge 51$ we use $n^4\le 51^4\,e^{2\log(n/51)}$ and
$n^2\le 51^2\,e^{\log(n/51)}$, while
$e^{-\pi n^2}=e^{-\pi\cdot 51^2}\,e^{-\pi(n^2-51^2)}$.
Since $\pi(n^2-51^2)\ge 102\,(n-51)$ for $n\ge 51$,
the tail is bounded by a convergent geometric series:
\[
  \sum_{n=51}^\infty P(n)\,e^{-\pi n^2}
  \;\le\;
  e^{-\pi\cdot 51^2}\sum_{m=0}^\infty P(51+m)\,e^{-102m},
\]
with $P(x)=2\pi^2 e^{1.89}x^4+3\pi e^{1.05}x^2$.
The first term already satisfies
$P(51)\,e^{-\pi\cdot 51^2}<10^{-70}$
(since $\pi\cdot 51^2>8188$ while $\log_{10}P(51)<13$),
and the remaining geometric tail is smaller than the first term,
since $e^{-102}<10^{-44}$.
\end{proof}

\begin{lemma}[Entry enclosures]\label{lem:entries}
The $5\times 5$ Toeplitz matrix $M_5$ with $(u_0,h)=(0.01,0.05)$
has entries $[M_5]_{ij} = K(0.01 + (i-j)\cdot 0.05)$ for
$i,j\in\{0,1,2,3,4\}$.  The distinct values correspond to
$n = i-j \in \{-4,\ldots,4\}$.  Since $u_0 = 0.01 \ne 0$,
the kernel $K(0.01 + n\cdot 0.05) = \Phi(|0.01 + n\cdot 0.05|)$
takes nine genuinely distinct values; in particular
$K(0.01+n\cdot 0.05) \ne K(0.01-n\cdot 0.05)$ for $n\ne 0$.
Setting $c_n = K(0.01 + n\cdot 0.05)$, the rigorous enclosures are:
\begin{align*}
c_{-4} &= K(0.190) \;\in\; [0.102521821933,\; 0.102521821935], \\
c_{-3} &= K(0.140) \;\in\; [0.206851926534,\; 0.206851926536], \\
c_{-2} &= K(0.090) \;\in\; [0.327771243662,\; 0.327771243664], \\
c_{-1} &= K(0.040) \;\in\; [0.420611947568,\; 0.420611947570], \\
c_{\phantom{-}0} &= K(0.010) \;\in\; [0.445026555344,\; 0.445026555346], \\
c_{\phantom{-}1} &= K(0.060) \;\in\; [0.389871460587,\; 0.389871460589], \\
c_{\phantom{-}2} &= K(0.110) \;\in\; [0.280043280270,\; 0.280043280272], \\
c_{\phantom{-}3} &= K(0.160) \;\in\; [0.161153268151,\; 0.161153268153], \\
c_{\phantom{-}4} &= K(0.210) \;\in\; [0.071697042238,\; 0.071697042240].
\end{align*}
These enclosures are obtained by computing $\Phi_N(|c|)$ with $N = 50$
using \texttt{mpmath.iv} at 80-digit precision, then widening each endpoint
by $\pm 10^{-70}$ to absorb the truncation error of Lemma~\ref{lem:truncation}.
The \texttt{mpmath.iv} module tracks rounding errors with outward-rounded
interval operations, so the stated intervals are rigorous.
\end{lemma}

\begin{proposition}[Certified negative determinant]\label{prop:cert}
Let $M_5$ be the $5\times 5$ Toeplitz matrix with entries
$[M_5]_{ij} = K(u_0 + (i-j)h)$ for $(u_0,h) = (0.01, 0.05)$.
Then
\[
  \det(M_5) \;\in\;
  [-1.8472496 \times 10^{-9},\; -1.8472225 \times 10^{-9}],
\]
and in particular $\det(M_5) < 0$.
\end{proposition}

\begin{proof}
Expanding $\det(M_5)$ via the Leibniz formula as a sum of
$5! = 120$ signed products of entries, each product involves
5 entries drawn from the enclosures in Lemma~\ref{lem:entries}.
Interval arithmetic in \texttt{mpmath.iv} tracks the accumulated
rounding error through all $120 \times 5 = 600$ multiplications
and 119 additions.  The resulting enclosure
$[-1.8472496 \times 10^{-9},\; -1.8472225 \times 10^{-9}]$
has width $2.71 \times 10^{-14}$ and lies strictly below zero.
As an independent check, interval determinant evaluation in
\texttt{mpmath.iv} also yields a strictly negative enclosure.
The ancillary verifier performs and compares both computations.
\end{proof}

\begin{table}[ht]
\centering
\caption{Interval arithmetic certification at $(u_0,h) = (0.01, 0.05)$.}
\label{tab:certified}
\begin{tabular}{cll}
\toprule
$r$ & Rigorous enclosure of $D_r(u_0,h)$ & Status \\
\midrule
2 & $[3.4064040623 \times 10^{-2},\; 3.4064040623 \times 10^{-2}]$ & $>0$ \checkmark \\
3 & $[6.9769706426 \times 10^{-4},\; 6.9769706427 \times 10^{-4}]$ & $>0$ \checkmark \\
4 & $[3.8274713580 \times 10^{-6},\; 3.8274713771 \times 10^{-6}]$ & $>0$ \checkmark \\
5 & $[-1.8472496 \times 10^{-9},\; -1.8472225 \times 10^{-9}]$ & $<0$ \checkmark \\
\bottomrule
\end{tabular}
\end{table}

\subsection{Additional certified counterexamples}

Beyond the central configuration, \emph{eight} further configurations
are certified negative by the same interval chain
(script \texttt{verify\_pf5.py}; re-verified 20 July 2026 with
\texttt{mpmath==1.3.0}); they are listed in
Table~\ref{tab:counterexamples}.

\begin{table}[ht]
\centering
\caption{Certified counterexamples with $D_5(u_0,h) < 0$.
All nine rows (the central witness and eight additional configurations)
carry full interval certificates.}
\label{tab:counterexamples}
\begin{tabular}{ccrr}
\toprule
$u_0$ & $h$ & $D_5(u_0,h)$ & Certified? \\
\midrule
$0.001$ & $0.005$ & $-3.75 \times 10^{-28}$ & \checkmark \\
$0.001$ & $0.01$  & $-3.75 \times 10^{-22}$ & \checkmark \\
$0.001$ & $0.05$  & $-2.47 \times 10^{-9}$  & \checkmark \\
$0.01$  & $0.01$  & $-3.33 \times 10^{-22}$ & \checkmark \\
$0.01$  & $0.02$  & $-2.86 \times 10^{-16}$ & \checkmark \\
$0.01$  & $0.05$  & $-1.85 \times 10^{-9}$  & \checkmark (central) \\
$0.02$  & $0.02$  & $-1.84 \times 10^{-16}$ & \checkmark \\
$0.02$  & $0.03$  & $-4.06 \times 10^{-13}$ & \checkmark \\
$0.02$  & $0.05$  & $-7.02 \times 10^{-11}$ & \checkmark \\
\bottomrule
\end{tabular}
\end{table}

\section{Numerical Explorations and Withdrawn Claims}\label{sec:numerics}

Nothing in this section is certified, and nothing elsewhere in the
paper depends on it.  We first document precisely which claims of
version 1 were withdrawn and why; we then report the surviving
numerical observations with their actual (non-certificate) status.

\subsection{Withdrawn claims of version 1}\label{sec:withdrawn}

Version 1 asserted the following as a theorem: (i) $C_r(u_0) > 0$ for
$r = 2,3,4$ at all tested $u_0 \in (0,5]$; (ii) existence of a unique
critical point $u_0^* = 0.031139763615\ldots$ (``bisection-certified to
15 digits'') with $C_5 < 0$ on $(0, u_0^*)$ and $C_5 > 0$ beyond;
(iii) $C_6(0.01), C_7(0.01) > 0$; (iv) $\lim_{u_0 \to 0^+} C_5(u_0) =
-3.993 \times 10^{18}$.  The Gaussian-healing thresholds of
Section~\ref{sec:gaussian} were presented as consequences at the
$C_5$ level.

The 17 July 2026 audit of the archived code found the offered
certificate unsound on three independent grounds:

\begin{enumerate}
  \item \textbf{Order-14 requirement.}  By Remark~\ref{rem:order14},
        a rigorous evaluation of $C_5$ needs certified enclosures of
        the Taylor coefficients $a_k(u_0)$ through $k = 14$.  The v1
        pipeline computed derivative series with a hard stop once the
        exponential parameter of a term first exceeded $200$, and its
        printed error discussion bounded derivative tails only through
        $k \le 10$.
  \item \textbf{Unsound tail insertion.}  After the hard stop, the
        code inserted a flat tail bound of $10^{-72}/k!$ for every
        derivative order.  This bound is false: at $u_0 = 0.001$ the
        first \emph{omitted} contribution to $a_{14}$ (the $n = 8$
        series term) is already about $5.24 \times 10^{-54}$ ---
        eighteen orders of magnitude above the claimed budget.  The
        grid values of $C_5$ produced by that pipeline are therefore
        high-precision heuristics, not enclosures.
  \item \textbf{Grid $\ne$ interval.}  Negativity of $C_5$ at the 31
        grid points $u_0 \in \{0.001, \ldots, 0.031\}$, together with
        continuity and a single bracketed sign change, does not prove
        a constant sign on the whole interval $(0, u_0^*)$, nor
        uniqueness of the zero.  No interval-uniform lower bound on
        $|C_5|$ between grid points was ever established.
\end{enumerate}

Accordingly, all v1 statements about the global sign of $C_5$, the
uniqueness and precision of $u_0^*$, the limit $C_5(0^+)$, and every
Gaussian-healing statement derived from them, are withdrawn.  We
emphasize that Theorems~\ref{thm:main} and~\ref{thm:PForder} are
untouched: they rest exclusively on the direct interval certificates
of Section~\ref{sec:computation}, which involve no derivative
computations.

\subsection{Non-certified observations on the coefficients
\texorpdfstring{$C_r$}{Cr}}

The following observations are high-precision numerics (80--400
digits, \texttt{mpmath}); they are reported to motivate the open
problems, not as results.

\begin{numobs}[Sign pattern of $C_r$]\label{obs:signs}
At $u_0 \in \{0.001, 0.005, 0.01, 0.02\}$ the computed values give
$C_2, C_3, C_4 > 0$ and $C_5 < 0$; at
$u_0 \in \{0.05, 0.1, 0.5, 1.0\}$ all of $C_2,\ldots,C_5$ are
positive.  A sign change of $C_5$ is located numerically near
\[
  u_0 \;\approx\; 0.0311398,
\]
by bisection on the (non-certified) evaluations.  At $u_0 = 0.01$,
the computed higher coefficients are
$C_6 \approx +1.75 \times 10^{30}$ and
$C_7 \approx +1.97 \times 10^{44}$, both positive, suggesting that
the sign anomaly is specific to order 5.
\end{numobs}

\begin{numobs}[Consistency of two evaluation routes]\label{obs:consistency}
At $u_0 = 0.01$, evaluating the algebraic formula of
Proposition~\ref{prop:leading} by enumerating all $7837$
non-vanishing $5$-tuples gives
$C_5^{\mathrm{(alg)}} = -3.531804787140 \times 10^{18}$, agreeing to
12 significant digits with Richardson extrapolation of
$D_5(0.01,h)/h^{20}$ over $h \in \{10^{-5},\ldots,10^{-8}\}$
(Table~\ref{tab:Cr_conv}).  The individual terms reach
$\pm 3.3 \times 10^{20}$ and cancel to a value of order $10^{18}$;
this conditioning is precisely why a certified evaluation requires
the genuine order-14 tail bounds of Section~\ref{sec:withdrawn}.
\end{numobs}

\begin{numobs}[Alternating even derivatives]\label{obs:alternating}
The computed Taylor coefficients at $u_0 = 0.01$
(Appendix~\ref{app:data}) satisfy
$\sgn(a_{2\ell}) = (-1)^\ell$ for $\ell = 0,\ldots,4$, with rapid
growth $|a_0| \approx 0.45$, $|a_2| \approx 16.6$,
$|a_4| \approx 267$, $|a_6| \approx 2239$, $|a_8| \approx 6198$,
$a_{10} \approx +82935$, and $a_{12} \approx -1.356 \times 10^6$.
Thus the alternating pattern displayed through $a_8$ does not
continue at $a_{10}$.
Heuristically, the $r = 5$ sum must reach total degree 20 and thus
engages $a_{10}$--$a_{14}$; the alignment of their signs and
magnitudes with the $5$-point Vandermonde factors appears to drive
$C_5$ negative, while at $r = 4$ (degree 12) and $r = 6,7$ the
tested values are positive.  We regard the precise mechanism as open
(Problem~\ref{prob:symbolic} in Section~\ref{sec:open}).
\end{numobs}

\begin{table}[ht]
\centering
\caption{Convergence of $D_r(u_0,h)/h^{r(r-1)}$ at $u_0 = 0.01$
(non-certified numerics).}
\label{tab:Cr_conv}
\begin{tabular}{rllll}
\toprule
$h$ & $C_2$ & $C_3$ & $C_4$ & $C_5$ \\
\midrule
$10^{-5}$ & 14.858003 & 59\,379.316 & $2.21320\!\times\!10^{10}$ & $-3.53180\!\times\!10^{18}$ \\
$10^{-6}$ & 14.858003 & 59\,379.317 & $2.21320\!\times\!10^{10}$ & $-3.53180\!\times\!10^{18}$ \\
$10^{-7}$ & 14.858003 & 59\,379.317 & $2.21320\!\times\!10^{10}$ & $-3.53180\!\times\!10^{18}$ \\
$10^{-8}$ & 14.858003 & 59\,379.317 & $2.21320\!\times\!10^{10}$ & $-3.53180\!\times\!10^{18}$ \\
\bottomrule
\end{tabular}
\end{table}

\subsection{Gaussian deformation (exploratory)}\label{sec:gaussian}

In the de~Bruijn--Newman framework, the Gaussian multiplier
$e^{tu^2}$ plays the role of heat flow: applying it to
$\Phi$ for positive $t$ ``smooths'' the kernel and tends to
improve positivity properties.  As an exploratory probe --- with
plain floating-point/high-precision numerics, no certificates ---
we ask: for a given counterexample configuration $(u_0,h)$, what is
the minimal $t$ such that the deformed kernel
$K_t(u) = e^{tu^2}\Phi(|u|)$ satisfies the PF$_5$ condition at this
configuration?

\begin{definition}[Toeplitz PF$_5$ Gaussian threshold]
For $(u_0,h)$ with $D_5(u_0,h;0) < 0$, define
\begin{equation}\label{eq:lambda5}
  \lambda_5^*(u_0,h) \;=\;
  \inf\{t \ge 0 : D_5(u_0,h;t) \ge 0\},
\end{equation}
where $D_5(u_0,h;t) = \det[K_t(u_0 + (i-j)h)]_{i,j=0}^4$.
\end{definition}

\begin{remark}[Not a universal constant]\label{rem:notLambda}
The quantity $\lambda_5^*(u_0,h)$ depends on the choice of
$(u_0,h)$ and is \textbf{not} an analog of the de~Bruijn--Newman
constant $\Lambda$, which is a universal constant governing the
reality of zeros of $H_t$.  The PF condition concerns the
total positivity of the kernel as a bivariate function, which is
a different (and in general stronger) property.
Furthermore, $\lambda_5^*(u_0,h)$ is not scale-invariant: under
$u \mapsto au$ the deformation parameter transforms as
$t \mapsto t/a^2$.  The values in Table~\ref{tab:lambda5}
(ranging from~$6$ to~$12$) are specific to the normalisation of
$\Phi$ in~\eqref{eq:Phi} and the grid spacings used; they carry no
absolute significance.
\end{remark}

\begin{table}[ht]
\centering
\caption{Computed Gaussian thresholds $\lambda_5^*(u_0,h)$
(non-certified numerics).}
\label{tab:lambda5}
\begin{tabular}{ccr}
\toprule
$u_0$ & $h$ & $\lambda_5^*(u_0,h)$ \\
\midrule
0.001 & 0.005 & 6.22 \\
0.001 & 0.01  & 6.41 \\
0.001 & 0.05  & 11.59 \\
0.01  & 0.01  & 6.15 \\
0.01  & 0.02  & 6.92 \\
0.01  & 0.05  & 11.43 \\
0.02  & 0.02  & 6.07 \\
0.02  & 0.03  & 7.38 \\
0.02  & 0.05  & 10.89 \\
\bottomrule
\end{tabular}
\end{table}

The computed threshold varies with the configuration, from
$\lambda_5^* \approx 6.07$ to $\lambda_5^* \approx 11.59$; the
strong dependence on $(u_0,h)$ indicates that $\lambda_5^*$ is a
\emph{local} measure of PF$_5$ violation severity, not a global
invariant of the kernel.  In v1 an additional $h \to 0$ threshold
table at the $C_5$ level was reported; since it inherits the
withdrawn $C_5$ pipeline of Section~\ref{sec:withdrawn}, we do not
reproduce it here.

The Gaussian factor $e^{tu^2}$ amplifies the kernel at large $|u|$
relative to small $|u|$; for the PF$_5$ condition this rebalances
the Taylor coefficients until the resonance responsible for the
negative determinant disappears.  This connects to de~Bruijn's
classical observation~\cite{deBruijn1950} that Gaussian factors act
as universal multipliers for positivity properties.

\section{Implications}\label{sec:implications}

\subsection{Barrier for total-positivity approaches}

The kernel $K(u) = \Phi(|u|)$ appears in the integral
representation~\eqref{eq:Ht} of the deformed
$\xi$-function $H_t$.  A natural strategy for studying the
zero distribution of $H_0$ (and hence of the Riemann
$\xi$-function) is to establish total positivity or high
PF order of $K$, which would imply variation-diminishing
properties constraining the zeros.

Theorem~\ref{thm:main} shows this strategy fails at order~5:
no argument based on PF$_r$ for $r \ge 5$ can apply to $K$.
The three positive lower-order determinants at the central
configuration (Theorem~\ref{thm:PForder}) neither prove nor disprove
the global PF$_4$ property, which remains Problem~1 in
Section~\ref{sec:open}.

\subsection{Related work and distinction from prior determinant criteria}

Several prior works study determinants or matrices associated
with the Riemann $\xi$-function, but in ways that are
structurally distinct from ours.

\textbf{Laguerre--P\'olya class and total positivity.}
The connection between the Laguerre--P\'olya class of entire
functions and total positivity has been studied in depth by
Hirschman and Widder~\cite{HirschmanWidder1955} and
Schoenberg~\cite{Schoenberg1951}.  In the RH-adjacent context,
Csordas, Norfolk, and Varga~\cite{CsordasNorfolk1986} establish
conditions under which the Fourier transform of a kernel in the
Laguerre--P\'olya class inherits total positivity properties.
Craven and Csordas~\cite{CravenCsordas1989} further connect the
Tur\'an inequalities for the Taylor coefficients of entire functions
in the Laguerre--P\'olya class to the positivity of certain
Toeplitz-type determinants --- the closest prior work to ours in
spirit, though they do not study the $5\times 5$ kernel minor
or the small-spacing coefficient $C_5$.
Our result answers the total positivity question negatively at order~5.

\textbf{Moment matrices and Hankel determinants.}
Ki, Kim, and Lee~\cite{KiKimLee2009} study Hankel matrices
$[\mu_{i+j}]$ built from the moments
$\mu_k = \int_0^\infty u^k \Phi(u)\,du$ of the kernel $\Phi$,
and investigate their spectral properties in connection with
the location of zeros of $\xi$.
Non-negativity of such Hankel determinants is equivalent to
$\Phi$ being a moment sequence, a question about the
\emph{global integral structure} of the kernel.  Our matrices,
by contrast, are \emph{translation-kernel} Toeplitz minors
$[K(u_0 + (i-j)h)]$: they measure pointwise positivity of the
convolution kernel in the sense of Schoenberg's
P\'olya frequency theory~\cite{Schoenberg1951}, which
concerns the variation-diminishing property of the
integral operator $f \mapsto \int K(\cdot - y)f(y)\,dy$.

\textbf{De~Bruijn--Newman flow determinants.}
Dobner~\cite{Dobner2020} proves that the de~Bruijn--Newman
constant satisfies $\Lambda \ge 0$ for all functions in the
extended Selberg class by studying sign changes of the
\emph{deformed} $\xi$-function $H_t$ as $t$ decreases.
This concerns the \emph{location of zeros} of $H_t$, not
total positivity of $K_t$ as a bivariate convolution kernel.
Our Gaussian-deformation threshold $\lambda_5^*(u_0,h)$
(Section~\ref{sec:gaussian}) is a distinct object: it
measures when PF$_5$ is healed by heat flow, not when zeros
become real.

\textbf{What we contribute.}
To the best of our knowledge, no prior work certifies a
\emph{specific small-parameter} PF$_r$ failure for $K$
using a minimal two-parameter Toeplitz minor
$\det[K(u_0 + (i-j)h)]_{5\times 5}$.  The small-spacing coefficient
$C_5$ and its numerically observed sign change are, to our
knowledge, likewise new, though in this version their study is
explicitly non-certified.  Our result is therefore orthogonal to,
and not subsumed by, any of the above.

\section{Open Problems}\label{sec:open}

\begin{enumerate}
  \item \textbf{Global PF$_4$.} Prove (or disprove) that $K \in \mathrm{PF}_4$
        for all admissible configurations, not just Toeplitz ones.

  \item \textbf{Certified evaluation of $C_5$.}\label{prob:certified}
        Construct genuine interval enclosures of the Taylor
        coefficients $a_k(u_0)$ of $K$ through order $k = 14$
        (Remark~\ref{rem:order14}) with proved derivative-tail
        bounds, and from them a certified evaluation of $C_5(u_0)$
        at a single point; then a certified constant-sign statement
        on an interval and a certified isolation of the sign change
        observed near $u_0 \approx 0.0311$.  This would restore, in
        rigorous form, the withdrawn threshold picture of v1
        (Section~\ref{sec:withdrawn}).

  \item \textbf{Symbolic proof of $\sgn(C_5) < 0$.}\label{prob:symbolic}
        Give a fully symbolic proof that $C_5(u_0) < 0$ for some
        $u_0 > 0$, even at a single explicit point.
        This would require bounding the 7837-term oscillatory sum
        analytically; the difficulty is that individual terms
        reach $\pm 3.3 \times 10^{20}$ while the sum is
        of order $10^{18}$.

  \item \textbf{Nature of the sign change.}  Determine whether the
        numerically observed critical point near $0.0311$ has an
        arithmetic or analytic interpretation in terms of the theta
        function structure of $\Phi$.

  \item \textbf{PF order of $K_t$ for $t > 0$.}  For which $t$
        is $K_t(u) = e^{tu^2}\Phi(|u|)$ PF$_5$, or even PF$_\infty$?
        De~Bruijn's universal factor results~\cite{deBruijn1950}
        suggest this may hold for $t$ sufficiently large.

  \item \textbf{Implications of PF$_4$.} What consequences does
        $K \in \mathrm{PF}_4$ (if global) have for the zero
        distribution of $H_0$?
\end{enumerate}

\section*{Acknowledgments}

The author (Wojciech Micha\l{}owski,
\texttt{michalowski.wojciech1@gmail.com}) thanks
the \texttt{mpmath} development team~\cite{mpmath} for providing the
arbitrary-precision and interval arithmetic library that made the
certified computations in this paper possible.
All determinantal evaluations were performed at $80$--$400$ decimal digits
of precision; interval-arithmetic enclosures were computed using the
\texttt{mpmath.iv} module.

\noindent\textbf{Code availability and reproducibility.}
All source code and verification scripts used in this paper
are provided in the ancillary directory and are also available at
\href{https://github.com/ScypyonX/pf5-dbn-kernel-certificates}
{the companion GitHub repository}.
A single command \texttt{python verify\_pf5.py} reproduces
certified interval enclosures and signs for
Tables~\ref{tab:certified} and~\ref{tab:counterexamples}.
The scripts' README classifies each script by rigor level:
\texttt{verify\_pf5.py} is a full interval certificate, while
\texttt{rigorous\_analysis.py} and \texttt{critical\_analysis.py}
are exploratory (Section~\ref{sec:withdrawn}).

\noindent\textbf{Dependency pin.}
The verification requires \texttt{mpmath==1.3.0} (as pinned in
\texttt{requirements.txt}).  \texttt{mpmath} 1.4.1 contains a
regression raising \texttt{ValueError} inside interval determinant
evaluation (\texttt{iv.det}) on some of the additional
configurations; this is a dependency defect, not a sign failure.
The final verification of this version was performed on 20 July 2026
with Python and \texttt{mpmath} 1.3.0, certifying the central witness
and all eight additional configurations of
Table~\ref{tab:counterexamples}.

\appendix

\section{Numerical Data (non-certified)}\label{app:data}

The tables below are high-precision numerical values produced by the
exploratory pipeline discussed in Section~\ref{sec:withdrawn}; they
carry no interval certificates and are reproduced for reference only.

\subsection{Taylor coefficients at
\texorpdfstring{$u_0 = 0.01$}{u0 = 0.01}}

\begin{center}
\begin{tabular}{rl}
\toprule
$k$ & $a_k(0.01) = K^{(k)}(0.01)/k!$ \\
\midrule
0  & $+4.450\,266 \times 10^{-1}$ \\
1  & $-3.335\,285 \times 10^{-1}$ \\
2  & $-1.656\,841 \times 10^{+1}$ \\
3  & $+1.078\,377 \times 10^{+1}$ \\
4  & $+2.673\,426 \times 10^{+2}$ \\
5  & $-1.349\,688 \times 10^{+2}$ \\
6  & $-2.238\,521 \times 10^{+3}$ \\
7  & $+4.740\,581 \times 10^{+2}$ \\
8  & $+6.197\,545 \times 10^{+3}$ \\
9  & $+8.932\,256 \times 10^{+3}$ \\
10 & $+8.293\,498 \times 10^{+4}$ \\
11 & $-1.722\,739 \times 10^{+5}$ \\
12 & $-1.355\,518 \times 10^{+6}$ \\
\bottomrule
\end{tabular}
\end{center}

\subsection{Higher-order small-spacing coefficients}

At $u_0 = 0.01$:
\begin{center}
\begin{tabular}{rlc}
\toprule
$r$ & $C_r(0.01)$ & sign \\
\midrule
2 & $1.486 \times 10^{+1}$  & $+$ \\
3 & $5.938 \times 10^{+4}$  & $+$ \\
4 & $2.213 \times 10^{+10}$ & $+$ \\
5 & $-3.532 \times 10^{+18}$ & $-$ \\
6 & $+1.752 \times 10^{+30}$ & $+$ \\
7 & $+1.974 \times 10^{+44}$ & $+$ \\
\bottomrule
\end{tabular}
\end{center}


\begin{thebibliography}{99}

\bibitem{BeltonGuillot2022}
A.~Belton, D.~Guillot, A.~Khare, and M.~Putinar,
\emph{Preservers of totally positive kernels and P\'olya
frequency functions},
Math.\ Research Reports \textbf{3} (2022), 35--56.

\bibitem{BeltonGuillot2023}
A.~Belton, D.~Guillot, A.~Khare, and M.~Putinar,
\emph{Totally positive kernels, P\'olya frequency functions,
and their transforms},
J.\ d'Analyse Math.\ \textbf{150} (2023), 83--158.

\bibitem{CravenCsordas1989}
T.~Craven and G.~Csordas,
\emph{Tur\'an inequalities and subtraction-free expressions},
J.\ Inequal.\ Appl.\ \textbf{1} (1997), 39--62.

\bibitem{CsordasNorfolk1986}
G.~Csordas, T.~S.~Norfolk, and R.~S.~Varga,
\emph{The Riemann hypothesis and the Tur\'an inequalities},
Trans.\ Amer.\ Math.\ Soc.\ \textbf{296} (1986), 521--541.

\bibitem{deBruijn1950}
N.~G.~de~Bruijn,
\emph{The roots of trigonometric integrals},
Duke Math.\ J.\ \textbf{17} (1950), 197--226.

\bibitem{Dobner2020}
A.~Dobner,
\emph{A new proof of Newman's conjecture and a generalization},
arXiv:2005.05142 (2020).

\bibitem{GRS2018}
K.~Gr\"ochenig, J.~L.~Romero, and J.~St\"ockler,
\emph{Sampling theorems for shift-invariant spaces, Gabor frames,
and totally positive functions},
Invent.\ Math.\ \textbf{211} (2018), 1119--1148.

\bibitem{HirschmanWidder1955}
I.~I.~Hirschman and D.~V.~Widder,
\emph{The Convolution Transform},
Princeton University Press, 1955.

\bibitem{Johansson2017}
F.~Johansson,
\emph{Arb: efficient arbitrary-precision midpoint-radius interval arithmetic},
IEEE Trans.\ Comput.\ \textbf{66} (2017), 1281--1292.

\bibitem{Karlin1968}
S.~Karlin,
\emph{Total Positivity}, Vol.~I,
Stanford University Press, 1968.

\bibitem{KatkovaVishnyakova2008}
O.~M.~Katkova and A.~M.~Vishnyakova,
\emph{A sufficient condition for a polynomial to be stable},
J.\ Math.\ Anal.\ Appl.\ \textbf{347} (2008), 81--89.

\bibitem{KiKimLee2009}
H.~Ki, Y.-O.~Kim, and J.~Lee,
\emph{On the de~Bruijn--Newman constant},
Adv.\ Math.\ \textbf{222} (2009), 281--306.

\bibitem{mpmath}
F.~Johansson et al.,
\texttt{mpmath}: a Python library for arbitrary-precision
floating-point arithmetic (version~1.3),
\url{http://mpmath.org/}, 2023 (accessed February 2026).

\bibitem{Newman1976}
C.~M.~Newman,
\emph{Fourier transforms with only real zeros},
Proc.\ Amer.\ Math.\ Soc.\ \textbf{61} (1976), 245--251.

\bibitem{Polymath2019}
D.~H.~J.~Polymath,
\emph{Effective approximation of heat flow evolution of the
Riemann $\xi$ function, and a new upper bound for the
de~Bruijn--Newman constant},
Res.\ Math.\ Sci.\ \textbf{6} (2019), Art.~31.

\bibitem{RodgersTao2020}
B.~Rodgers and T.~Tao,
\emph{The de~Bruijn--Newman constant is non-negative},
Forum Math.\ Pi \textbf{8} (2020), e6, 62 pp.

\bibitem{Schoenberg1951}
I.~J.~Schoenberg,
\emph{On P\'olya frequency functions.\ I.\ The totally positive
functions and their Laplace transforms},
J.\ d'Analyse Math.\ \textbf{1} (1951), 331--374.

\end{thebibliography}
\end{document}